\documentclass[a4paper,11pt]{article}
\usepackage[utf8]{inputenc}
 
\usepackage{amsfonts,amssymb}
\usepackage{amsmath}  %
\usepackage{color}

\newcommand{\Rset}{\mathbb R}
\newcommand{\Cset}{\mathbb C}
\newcommand{\C}{\mathcal C}

\newcommand{\Nset}{\mathbb N}

\newcommand{\dd}{\mathrm d}

\newcommand{\e}{\mathrm e}
\newcommand{\A}{\mathcal A}
\newcommand{\R}{\mathcal R}
\newcommand{\B}{\mathcal B}
\newcommand{\F}{\mathcal F}
\newcommand{\K}{\mathcal K}
\newcommand{\real}{\mathrm{Re\,}}
\newcommand{\Ker}{\mathrm{Ker\,}}
\def\qed{\hbox{\vrule width1.0ex height1.ex}\vspace{2mm}}
\newenvironment{proof}{\noindent{\bf Proof.}}{\qed \vskip 0.0ex}
\def\eqref#1{(\ref#1)}
\newtheorem{theorem}{Theorem}[section]
\newtheorem{proposition}[theorem]{Proposition}
\newtheorem{definition}[theorem]{Definition}
\newtheorem{lemma}[theorem]{Lemma}
\newtheorem{corollary}[theorem]{Corollary}
\date{}
\begin{document}

\title{Exact null controllability, complete stabilizability and final observability: the case of neutral type systems
\thanks{Published in \textit{Int. J. Appl. Math. Comput. Sci}, \textbf{27}(3), 489--499, 2017. 
DOI: {10.1515/amcs-2017-0034}}}
\author{ \textbf{Rabah {Rabah}}
\footnote{IRCCyN,  
IMT Atlantique,  Mines-Nantes, 4 rue Alfred Kastler, BP 20722, 44307 Nantes, France, e-mail: \texttt{rabah@emn.fr}}
\and \textbf{Grigory {Sklyar}}
 \footnote{Institute of Mathematics, University of Szczecin,
70451 Szczecin, Wielkopolska 15, Poland,   e-mail: \texttt{sklar@univ.szczecin.pl}}
\and \textbf{Pavel {Barkhayev}} 
\footnote{Institute for Low Temperature Physics and Engineering of the NAS,
   47 Lenin Ave.,  61103 Kharkiv, Ukraine,
 e-mail: {\tt barkhayev@ilt.kharkov.ua}}}

\maketitle
 
\newenvironment{exemple}{\begin{exemp} \em}{ \end{exemp}}
\bibliographystyle{dcu}

\begin{abstract}
 For abstract linear systems in Hilbert spaces we revisit the problems of exact controllability and complete
 stabilizability (stabilizability with an arbitrary decay rate), the latter property being related to exact null controllability.
 We consider also the case when the feedback is not bounded. We obtain a characterization of complete stabilizability for 
neutral type systems. Conditions of exact null controllability for  neutral type systems are discussed. 
By duality, we obtain a result about continuous final observability. Illustrative examples are given.
\end{abstract}
\textbf{Keywords:}
Exact null controllability, complete stabilizability, final observability, neutral type system.
\maketitle
\section{Introduction}
Consider the controlled neutral type system 
\begin{equation}\label{eq:1}
\dot z(t) = A_{-1}\dot z(t-1) +
Lz_t(\cdot) +Bu(t),
\end{equation}
where 
$$
Lz_t(\cdot)= \int^0_{-1}\left [A_2(\theta)\dot z(t+\theta) + A_3(\theta) z(t+\theta)\right ]{\mathrm d}\theta,
$$
with $z(t) \in \Rset^n$, $u(t) \in \Rset^m$, and the matrices $A_{-1}$, $A_{2}$, $A_{3}$ and 
 $B$ are of appropriate dimensions. The elements of $A_{2}$ and  $A_{3}$ take values in $L_2(-1,0)$.

System $(\ref{eq:1})$ may be represented in a Hilbert space by the equation
$$
   \dot x(t)= {\mathcal A} x(t) + {\mathcal B}u(t),
$$
where ${\mathcal B}u=(Bu,0)$ and ${\mathcal A}$ is the infinitesimal generator of a $C_0$-semigroup  $\e^{\A t}$
given in the product space 
$$
M_2(-1,0;\Rset^n) \stackrel{\mathrm{def}}{=} \Rset^n\times L_2(-1,0;\Rset^n),
$$
noted shortly $M_2$,  and defined by
$$
{\mathcal A} x(t)=
\begin{pmatrix}
Lz_t(\cdot) \cr
 \frac{{\mathrm d} z_t(\theta)}{{\mathrm d}\theta}
\end{pmatrix}, \quad
  x(t)= \begin{pmatrix}
v(t)\cr  z_t(\cdot)
\end{pmatrix},
$$
with the domain $D({\A})$ given by
$$
D({\A})=\left \{
( v,  \varphi)
: \varphi(\cdot) \in
   H^1, v=\varphi(0)-A_{-1}\varphi(-1)\right\}.
$$

\medskip 

Our purpose is to analyze exact  null control\-la\-bility of delay systems of neutral 
type  (\ref{eq:1}), to show the relation with the complete stabilizability 
(exponential stabilizability with an arbitrary decay rate)
of the system and, by duality, to give condi\-tions of the exact  final observability of such a system with  an output
$y(t)=Cz(t)$ or $y(t)=Cz(t-1)$, where $y(t)$ takes values in $\Rset^p$.

The problem of exact controllability for systems of neutral type has been widely investigated.  References  
and important results for system (\ref{eq:1}) can be found in the work of
\cite{Rabah_Sklyar_2007}.
A simplification and impro\-vement of some details of the proofs are given in \cite{Rabah_Sklyar_Barkhayev_2016}. The duality
 with exact (continuous) obser\-vability is analyzed in \cite{Rabah_Sklyar_2016}. For the stab\-ili\-zability problem, after 
the first important works \cite{Pandolfi_1976,Oconnor_Tarn_1983}, there were many results on the stabilizability of
 delay systems (see, for example,  \cite{Richard_2003,Michiels_Niculescu_2007} and references therein) 
but neutral type systems have been  less investigated \cite{Pritchard_Salamon_1987,Salamon_Pitman_1984}. 
In \cite{Hale_Verduyn_2002} the main scheme of stabilizing neutral type systems and   the robustness, with respect to the delays,  
of the stabilizing feedback 
were  analyzed.  The problem of asymptotic non\-exponential  stabilizability, which appears only for neutral type systems, 
was treated in \cite{Rabah_Sklyar_Rezounenko_2008,Rabah_Sklyar_Barkhayev_2012}, such problem occurs for some  systems governed by partial differential 
equations (see for example  \cite{Sklyar_Szkiebiel_2013}).

\medskip

This paper is organized as follows. In Section~\ref{prel} we give results on the rel\-ation between 
exact null cont\-rol\-lability and comp\-lete stabil\-izability for abs\-tract systems in Hilbert spaces.  
In Section \ref{neutral}, we give necessary conditions of  exact null controllability and we characterize 
complete stabilizability for neutral type systems. 
Then we form\-ulate a conjecture on the equivalence between exact null controllability and complete stabili\-zability for neutral type systems.
Section \ref{observ} is concerned with the dual notion of 
observability: final continuous observability.
 \section{Preliminary results}\label{prel}
In this section we consider  the abstract system
\begin{equation}\label{eq:absys}
 \dot x = \A x + \B u
\end{equation}
where the linear  operator $\A$, with domain $D(\A)$, is the infinitesimal generator of a $C_0$-semigroup $\e^{\A t}$ in the Hilbert space 
$X$ and $\B$ is a linear operator, which
may be unbounded but admissible  (see, for example,  \cite{Tucsnak_Weiss_2009}), from the Hilbert space $U$ to $X$.
\subsection{Bounded input and feedback}
Let us first sup\-pose that the operator $\B$ is bounded. The solution of the system (\ref{eq:absys}) with the initial condition $x_0$ and the control 
$u(t) \in L_2^{\mathrm {loc}}(\Rset^+;U)$ is given by
$$
x(t)= \e^{\A t}x_0+ \int_0^t\e^{\A (t-\tau)}Bu(\tau) \dd \tau.
$$
The following notions are well known (see for example \cite{Curtain_Zwart_1995}).
\begin{definition}{}\label{ex_contr}
System (\ref{eq:absys}) is said to be exactly controllable at time $T$ if for all $x_0, x_1 \in X$, there is a control 
$u(t) \in L_2(0,T;U)$ such that the corresponding solution of the system verifies $x(T)=x_1$.
The system is said to be exactly null controllable if in the preceding definition $x_1=0$.
\end{definition}
There are several results about   exact (null) controllability. For example, it is well known that if $\B$ is compact, par\-ticularly
if $U$ is finite dimensional, then there is no exact controllability (first proved in \cite{Kuperman_Repin_1971}, see also \cite{Curtain_Zwart_1995} and references 
therein). Another condition of exact control\-lability, in the case of the bounded operator $\B$,
 is that for all $t\ge 0$,  the operator $\e^{\A t}$ is onto (surjective) \cite{Louis_Wexler_1983}. 

In what follows, we need the following criteria of exact (null) controllability \cite{Curtain_Zwart_1995}.
\begin{theorem}{} \label{th:excont}
 System (\ref{eq:absys}) is exactly null controllable at time $T$ if and only if
$$
\exists \delta >0 : \ \forall x \in X, \ \int_0^T\left \Vert B^* \e^{\A^* t}x\right \Vert^2 \dd t \ge \delta^2 \left \Vert \e^{\A^* T} x\right \Vert^2.
$$
For the condition of  exact controllability, the operator $\e^{\A^* T}$ must be replaced by the identity $I$  in the right part of the inequality.
\end{theorem} 
The characterization of exact null controllability is due to a result on range inclusion in 
Hilbert spaces \cite{Douglas_1966}.

We also need  some notions of stabilizability.
\begin{definition}{}\label{def:stabilizab}
 System (\ref{eq:absys}) is said to be exponentially stabilizable if there is a linear bounded feedback operator  $\F$ such that the semigroup 
$\e^{(\A +\B \F)t}$ is exponentially stable: there is a $\omega > 0$ such that
\begin{equation}    \label{exp-stab}
\left \Vert \e^{(\A +\B \F)t}\right\Vert \le M_{\omega}\e^{-\omega t}, \quad M_{\omega} \ge 1.
\end{equation}
The system is said to be  completely stabilizable (or stabilizable with an arbitrary decay rate) if for all $\omega >0$ there is 
a linear bounded feedback $\F_\omega$ such that (\ref{exp-stab}) holds.
\end{definition}

\medskip

The relation between exact controllability and stab\-ili\-zability is as follows:
exact null controllability implies exponential  stabilizability.  If $\e^{\A t}$ is a group, complete stabilizability  implies exact controllability as shown in
\cite[Theorem 3.4, p. 229]{Zabczyk_1992}. Note that the original proof was obtained by \cite{Zabczyk_1976} which extends the result 
obtained by \cite{Megan_1975}. 
The same result was proved in  \cite{Rabah_Karrakchou_1997,Zeng_Yi_Xie_2013} for the case of a semigroup $\e^{\A t}$ provided that 
the operators $\e^{\A t}$ are surjective for all $t\ge 0$.

We were tempted to extend  this latter result to ex\-act null control\-lability, possibly under some additional conditions. 
However the situation is not so simple. We have the following implication, but the converse is not true.
\begin{theorem}{}\label{th:Fbounded}
If system (\ref{eq:absys})    is exactly null controllable, then it is completely stabilizable by  a bounded feedback
$\F$. 
\end{theorem}
\begin{proof}{}
Suppose that the system is exactly null control\-lable at time $T$. Then 
$$
\forall x_0 \in X, \ \exists u(\cdot) \in L_2(0,T;U) : \quad x(T,x_0, u(\cdot)) =0,
$$
where $x(t) = x(t,x_0, u(\cdot)) $ is the solution with the initial condition $x_0$ and the control $u(t)$:
$$
 x(t,x_0, u(\cdot))= \e^{\A t}x_0 + \int_0^t \e^{\A (t-\tau)}\B u(\tau) \dd \tau.
$$
 Then for every $x_0 \in X$, there exists $u(\cdot) \in L_2(0,\infty;U)$ such that
$$
 \int_0^{+ \infty}\left ( \Vert x(t)\Vert^2 +  \Vert u(t)\Vert^2 \right)\dd t < \infty.
$$
This means that the system is exponentially stabilizable \cite[Th. 3.3, p. 227]{Zabczyk_1992} :
$$
\exists F_{\omega_0} \in \mathcal {L}(U,X) : \ \left \Vert \e^{(\A +\B \F_{\omega_0})t}\right \Vert \le M_{\omega_0}\e^{-\omega_0 t}, \ \omega_0 > 0.
$$
On the other hand, the  exact null controllability of  system  (\ref{eq:absys}) is equivalent to the  exact null 
controllability of the system
$$
\dot x = (\A +\omega I) x + \B u, \quad \omega>0.
$$
This means that for all $\omega>0$, for some  $\mu_\omega > 0$, there is $\F_{\omega} \in \mathcal {L}(U,X)$
such that
$$
 \left \Vert \e^{(\A +\B \F_{\omega})t}\right \Vert \le M_{\mu_\omega} \e^{-(\mu_\omega + \omega) t}
\le M_\omega \e^{- \omega t}.
$$
\end{proof}

In order to explain the fact that the converse is not true and that the situation is more complicated, 
we give examples of two systems without control, where the semigroups are exponentially   stable  with arbitrary decay rate, 
but where the states may or may not  reach the null state in finite time. These examples can be found in \cite{Rhandi_2002} 
in the spaces of continuous functions.
\subsection*{Example 1.}
In the space $L_2(0,+\infty)$, consider the semi\-group 
$$
S(t)f(x) = \e^{-\frac{t^2}{2}- x t} f(x+t), \quad  t\ge 0, \quad x \ge 0.
$$ 
It is not difficult to see that for this semigroup, for all $\omega>0$, there is a constant $M_\omega \ge 1$  
such that $\| S(t)\| \le M_\omega \e^{-\omega t}$. We have also $\sigma(S(t))= \{0\}$ and then the spectrum of the 
infinitesimal generator is empty.   On the other hand, there are initial  conditions $f$ such that $S(t)f \ne 0$ for any $t \ge 0$.
\subsection*{Example 2.}
In the space $L_2(0,1)$, consider the semi\-group 
$$
S(t)f(x) = 
\left \{
\begin{array}{cc}
 f(x+t) & 0\le t+x\le 1,  \cr
0 &  t+x > 1.
\end{array}
\right.
$$ 
It is not difficult to see that for this semigroup, for all $\omega>0$, there is a constant $M_\omega \ge 1$  
such that $\| S(t)\| \le M_\omega \e^{-\omega t}$. We have also $\sigma(S(t))= \{0\}$ and then the spectrum of the 
infinitesimal generator is empty.  But, for any initial function $f \in L_2(0,1)$, we have $S(t)f(x)=0$ for $t >2$. This means that
$S(t)=0$, for all $t >2$. Then, for any control operator $\B$, the corresponding system is 
exactly null controllable at time $T>2$ with the trivial control $u=0$.

\subsection{Unbounded input and feedback operators}
For some control systems, the input operator $\B$ may not be bounded and it is very restrictive to assume that
the feedback operator $\F$ is bounded.  For a general theory on systems with unbounded control and observation 
we refer to the paper \cite{Salamon_1987}. For the subclass of interest, which includes linear neutral type systems  
  we refer to \cite{Pritchard_Salamon_1987} and \cite{Curtain_L_T_Z_1997,Guo_Z_H_2003}).

As our final goal is to analyze exact null control\-labi\-lity and complete stabilizability for neutral type sys\-tems, 
we will now consider a wider context of systems with unbounded input and output operators. However, the situation is much more complicated, 
even if some extension may be considered.

Let $X_1$ be $D(\A)$ endowed with the graph norm noted $\|x\|_1$ and $X_{-1}$  be the completion of the space $X$ 
with respect to the resolvent norm 
$$
\|x\|_{-1}=\left \|(\lambda I- \A)^{-1}x\right\|_{X}, \quad \lambda \in \rho(A).
$$
 We have the following relation
$$
X_1 \subset X  \subset X_{-1},
$$
with continuous dense injections. 
\begin{definition}{}\label{admis}
 Let $\B$ be a linear operator,  bounded from the Hilbert space $U$ to $X$. 
We say that $\B$ is an admissible input operator for the semigroup $\e^{\A t}$ if there exists $t_1$ such that 
$$
\int_0^{t_1}\e^{\A (t_1-\tau)}\B u(\tau) \dd \tau \in X_1,
$$
and for some $\beta>0$
$$
\left \Vert \int_0^{t_1}\e^{\A (t_1-\tau)}\B u(\tau) \dd \tau \right \Vert_{X_1} \le \beta \Vert u \Vert_{L_2(0,t_1)}.
$$
\end{definition}
\begin{definition}{}
 Assume that  operator $\F$ is a linear operator,  bounded from $X_{1}$ to the Hilbert space $Y$.
We say that it is an admissible output operator for the semigroup $\e^{\A t}$ if there exists $t_1>0$ such that 
 for some $\alpha>0$
$$
\left \Vert \F \e^{\A (t_1-\tau)}x \right  \Vert_{L_2(0,t_1)} \le \alpha \Vert x \Vert_{X}, \quad x \in X_1.
$$
\end{definition}
Admissibility for some $t_1$ implies admissibility for all $t>0$  
(see for example  \cite{Curtain_L_T_Z_1997}). 
From the general result on the perturbation of semigroup from the Pritchard-Salamon class, we can deduce the
 following Cauchy formula for the perturbed semigroup $\e^{(\A+\B \F)t}$, for admissible input and output operators $\B$ and 
$\F$: 
\begin{equation}\label{perturb}
 \e^{(\A+\B \F)t}x = \e^{\A t}x+ \int_0^t  \e^{\A (t-\tau)} \B \F \e^{(\A+\B \F)\tau}x \dd \tau,
\end{equation}
for all $x \in X_1$. Moreover $\e^{(\A+\B \F)t}$ extends to  a $C_0$-semi\-group on $X$.   

This means that Definition~\ref{def:stabilizab} may be reformulated for an admissible input operator 
and admissible output feedback.
\begin{theorem}{}\label{th:Funbounded}
If the system (\ref{eq:absys})  with an admissible  operator $\B$  is completely stabilizable by  an 
admissible  $\A$-bounded feedbacks then it is completely stabilizable by a bounded linear feedback  $\F$.
\end{theorem}
\begin{proof}{}
 In \cite[Theorem 5.5]{Curtain_L_T_Z_1997} (see also \cite{Guo_Z_H_2003}), in a more general situation,
 it is shown that  system  (\ref{eq:absys}) with an admissible operator $\B$ is exponentially
stabilizable by admissible feedback (in $X_1$ and $X$) if and only if it is exponentially stabilizable by a bounded feedback. 
Hence, we can suppose without loss of generality, that in (\ref{perturb}) the operator $\F$ is bounded: $\F \in \mathcal L(X,U)$. 
This means that complete stabilizability by admissible feedbacks holds if and only if there is complete stabilizability 
by bounded feedbacks.  
\end{proof} 
From this and Theorem~\ref{th:Fbounded} we can expect to extend the result of Theorem~\ref{th:Fbounded} for the case of unbounded  control and feedback. But, 
Theorem~\ref{th:Fbounded} is based on \cite[Th. 3.3, 227]{Zabczyk_1992} which needs  \cite[Th. 4.3, 240]{Zabczyk_1992} based on the assumption 
of exact null controllability given by Definition~\ref{ex_contr}. For the case of unbounded control and feedback, 
we refer to \cite[Theorem 3.3, page 132]{Pritchard_Salamon_1987}. The condition H4 used in this theorem is guaranteed by exact null  controllability in $X_{-1}$ (each initial 
state from $X_{-1}$ may be moved to zero by an $L_2$ control). 
\begin{corollary}{}\label{cor:1}
If system (\ref{eq:absys})  with admissible operator $\B$  is exactly null controllable in $X_{-1}$, then it is completely stabilizable by  an admissible  feedback
and then by a bounded feedback $\F$. 
\end{corollary}

\subsection{A technical Lemma}
In the next section we need  the following  lemma.
\begin{lemma}{}\label{lemma}
 Let $A$ be a $(n\times n)$-matrix and $B$ a $(n\times m)$-matrix. The following statements are equivalent.
\begin{enumerate}
\item For all $\lambda \in \Cset, \ \lambda \ne 0$, $\mathrm{rank}\begin{pmatrix}
                    \lambda I -A & B
                   \end{pmatrix} = n$.
\item \label{inclu}The following equality holds
\begin{eqnarray*}
\mathrm{rank} \begin{pmatrix}
B & AB & \cdots & A^{n-1}B  \end{pmatrix} 
=
 \mathrm{rank} \begin{pmatrix}
B & AB & \cdots & A^{n-1}B & A^{n}
\end{pmatrix}, 
\end{eqnarray*} and this  is equivalent to the inclusion:
$$
\mathrm{Im\,}A^{n} \subset \mathrm{Im}\begin{pmatrix}
B & AB & \cdots & A^{n-1}B  \end{pmatrix}.
$$
\end{enumerate}
\end{lemma}
\begin{proof}{}
Conditions 1 and 2 of the lemma  may be form\-ulated as follows.
\begin{enumerate}
 \item[1.] If there is $x \ne 0$ such that $A^*x=\lambda x$ 
and $B^*x=0$, then $\lambda =0$.
\item[2.] If $x \ne 0$ is such that $B^* A^{* i}x =0$, $i \in \Nset$, then $A^{*n}x=0$.
\end{enumerate}

Suppose that 1 holds.   
Let $\mathcal N$ be the subspace 
$$
\mathcal{N}= \{x : B^*A^{* i}x=0, \ i \in \Nset \}.
$$  
 It is easy to see that $\mathcal{N}$ is
 $A^*$-invariant and contained in $\Ker B^*$. The spectrum of the restriction of $A^*$ to  $\mathcal{N}$  is $\{0\}$ by Condition 1. This means that
$A^*$ is nilpotent in $\mathcal{N}$. As the dimension of $\mathcal{N}$ is $k \le n$, we obtain $A^{*n}x = 0$ for all $x \in \mathcal{N}$. 
This gives 2. 

 Let us show the equivalence of these conditions. Suppose now that 2 holds. Let $x \ne 0$ be such that $A^*x=\lambda x$ 
and $B^*x=0$. This implies that $B^* A^{* i}x = \lambda^i B^* x=0$, for all $i \in \Nset$. From Condition 2, we obtain that
$$
0=A^{*n}x = \lambda^n x.
$$ 
As $x \ne 0$, this implies  $\lambda =0$. This gives that statement 1 is verified.
 \end{proof}
\section{The neutral type system: controllability and stabilizability}\label{neutral}
In this section we analyze exact  null control\-la\-bility and the complete stabilizability 
(exponential stabilizability with an arbitrary decay rate) of delay system of neutral 
type  (\ref{eq:1}) and investigate the relation between the two notions.
By duality, we give condi\-tions of the exact  final observability of such system with  outputs
$$
y(t)=Cz(t) \qquad \mbox{or}\qquad y(t)=Cz(t-1),
$$
where $y(t)$ takes values in $\Rset^p$.

The relation between exact controllability and expo\-nential stabilizability for linear neutral type systems may
be found in several papers 
(see,  for example \cite{Salamon_Pitman_1984,Ito_Tarn_1985,Oconnor_Tarn_1983,Dusser_Rabah_2001}
and references therein).

For the analysis of  stabilizability,  we need the struc\-ture of the spectrum of the state operator $\A$
of  system (\ref{eq:1}) and the condition of the growth of semigroup $\e^{\A t}$.

\begin{theorem}{}{\em \cite{Rabah_Sklyar_Rezounenko_2005}} \label{prop5}
Let $ \Delta_{\A}$ be the matrix:
\begin{eqnarray*}
\lefteqn{\Delta_{\A} (\lambda)=}\\
&& \lambda I
- \lambda e^{-\lambda}A_{-1} - 
\int^0_{-1}\left [\lambda e^{\lambda s} A_2(s)  +e^{\lambda
s} A_3(s)\right]{\mathrm d}s.
\end{eqnarray*}
The spectrum of $\A$, noted $\sigma(\A)$,
consists of eigenvalues only which are the roots of the
equation $\det \Delta_{\A} (\lambda)=0$. 
The corresponding eigenvectors of ${\mathcal A}$ are of the form
$$
\begin{pmatrix}
v-e^{-\lambda}A_{-1}v \cr e^{\lambda\theta} v
\end{pmatrix}, \quad v\in \Ker  \Delta_{\A} (\lambda).
$$ 
The spectrum of $\A$ contains a non empty set of point of the form 
$$
\{\ln |\mu|+ {\rm i}(\arg \mu +2\pi
k)+{O}(1/k), \ 
k \in \mathbb{Z}\},
$$
where $\mu$ is a non-zero eigenvalue of the matrix $A_{-1}$.
\end{theorem}
The spectrum is countable and the semigroup $\e^{\A t}$ verifies the spectrum growth assumption (see, for example, \cite{Hale_Verduyn_1993}):
$$
\forall \omega > \omega_0= \sup \real \sigma (\A), \quad \exists M_\omega : \ \Vert \e^{\A t} \Vert \le 
M_\omega \e^{\omega t}.
$$
 \begin{definition}{}
 System~(\ref{eq:1}) is exactly null controllable if for some $T>0$ and 
for all $x_0 \in M_2$, there is a control $u(\cdot) \in L_2(0,T;\Rset^m)$ such that
$$
\e^{\A t}x_0 + \int_0^T  \e^{\A (T-\tau)} \B u(\tau)\dd \tau =0,
$$
this corresponds to the concept of complete controllability given first by N. N. Krasovski\v{\i} for retarded systems.
\end{definition}
Let $\R_T$ be the linear  operator defined by
$$
\R_T u(\cdot) = \int_0^T  \e^{\A (T-\tau)} \B u(\tau)\dd \tau,  u(\cdot) \in L_2(0,T;\Rset^m). 
$$
The operator $\R_T$  is bounded from $L_2(0,T;\Rset^m)$ to $X$. Moreover it takes values in
$D(\A)$ and is bounded from $ L_2(0,T;\Rset^m)$ to $X_1$ (see \cite[Corollary 2.7]{Ito_Tarn_1985} and 
\cite{Rabah_Sklyar_2007} for our system).

The exact null controllability may be formulated by the inclusion
$$
\mathrm{Im\,} \e^{\A T} \subset \mathrm{Im\,} \R_T,
$$
where $\mathrm{Im\,} \e^{\A T}$ and  $\mathrm{Im\,} \R_T$ are 
images of the operators  $\e^{\A T}$ and $\R_T$.
From the well-known characterization of range inclusion in Hilbert spaces \cite{Douglas_1966}
we can obtain the following prop\-osition, which is an extension of Theorem~\ref{th:excont}.
\begin{proposition}{}
System~(\ref{eq:1}) is exactly null controllable for some $T>0$ if and only if there is a  constant
$\delta >0$  such that 
$$
 \int_0^T \left \Vert \B^* \e^{\A^* (T-\tau)} x \right \Vert_{\Rset^m}^2 \dd \tau\ge \delta^2\left  \Vert \e^{\A^* T} x \right  \Vert_{M_2}^2,
$$
for all $x \in M_2$. 
\end{proposition}
We can now give the main result of this Section.
\begin{theorem}{}\label{th:main}
 If the system~(\ref{eq:1}) is exactly null controllable, then the following two  conditions hold
\begin{enumerate}
\item $\mathrm{rank}\begin{pmatrix}
                              \Delta_\A(\lambda) & B
                             \end{pmatrix} = n $ for all $\lambda \in \Cset$,
 \item $\mathrm{rank}\begin{pmatrix}
                              \mu I - A_{-1} & B
                             \end{pmatrix} = n $ for all $\mu \in \Cset$, $\mu \ne 0$.
\end{enumerate}
\end{theorem}
\begin{proof}{}
 Suppose that system~(\ref{eq:1}) is exactly null control\-lable. The necessity of condition~1 is trivial.
Let us show that condition~2 is verified. We follow a method used in \cite{Metelskiy_Minyuk_2006} (see also \cite{Khartovskii_Pavlovskaya_2013}).
Then, for some $T$, for all initial conditions, in particular for all  $\varphi \in H^1(-1,0; \Rset^n)$,
there is a control $u(\cdot) \in L_2(0,T;\Rset)$, $u(t)=0$ for $t > T$, such that $z(t)=0, \ t > T$.   We can suppose that $T > n$.

The function $z(t)$ is absolutely continuous and then almost everywhere differentiable. Then  we have 
$$
\dot z(t)= A_{-1}\dot z(t-1) + Lz_t + Bu(t).
$$
Replacing $\dot z(t-1)$ in this equation, we obtain
$$
\dot z(t)= A_{-1}(A_{-1} \dot z(t-2) +  Lz_{t-1}+ Bu(t-1) )+ Bu(t) .
$$
Without loss of generality one can suppose that the time $t$ is such that the function $u$ is well defined at these points. Repeating this procedure, we obtain
\begin{eqnarray*}
\lefteqn{\dot z(t)=A_{-1}^N\dot z(t-N)+ }\\&&\sum_{k=0}^{N-1}A_{-1}^{k}\left ( L z_{t-k} +
Bu(t-k)\right).
\end{eqnarray*}
Putting $t=N\ge T$, and using the continuity of $z(t)$ we obtain\\[1ex]
$0= A_{-1}^N\left (\dot z(+0) -\dot z(-0)\right) +$ 

\vskip -3ex

\begin{equation}\label{eq:3} 
\sum_{k=0}^{N-1}A_{-1}^{k}\left (  
Bu(N-k + 0)- Bu(N-k - 0)\right).
\end{equation}
As $z(t)$ for $t>0$ is the solution of equation (\ref{eq:1}) we have
$$
\dot z(+0) =A_{-1}\dot {z}(-1)+L z_{+0}(\cdot) +Bu(+0).
$$
Then, replacing this expression in (\ref{eq:3}), and putting the initial condition $z_0(\theta) = \varphi(\theta)$, we obtain 
\begin{eqnarray*}
\lefteqn{
A_{-1}^N\left (A_{-1}\dot\varphi(-1)+L \varphi(\theta)  -\dot \varphi(-0)\right) + A_{-1}^NBu(+0) +}\\&&
\hskip -1 ex \sum\limits_{k=0}^{N-1}A_{-1}^{k}\left (  
Bu(N-k + 0)- Bu(N-k - 0)\right)=0.
\end{eqnarray*}
As $ \dot \varphi(-0) \in \Rset^n$ may be chosen arbitrarily, we  obtain
$$
\mathrm{Im\,} A_{-1}^N \subset  \mathrm{Im}\begin{pmatrix}
B & A_{-1}B & \cdots & A_{-1}^{N-1}B  \end{pmatrix}.
$$
This may be written as
\begin{eqnarray*}
\lefteqn{\mathrm{rank} \begin{pmatrix}
B & A_{-1}B & \cdots & A_{-1}^{N-1} B  \end{pmatrix}=}\\&&
\mathrm{rank}
 \begin{pmatrix} B & A_{-1}B & \cdots & A_{-1}^{N-1}B & A_{-1}^N  \end{pmatrix}.
\end{eqnarray*}
By the Cayley-Hamilton theorem, this gives
 \begin{eqnarray*}
\lefteqn{\mathrm{rank} \begin{pmatrix}
B & A_{-1}B & \cdots & A_{-1}^{n-1} B  \end{pmatrix}=}\\&&
\mathrm{rank}
 \begin{pmatrix} B & A_{-1}B & \cdots & A_{-1}^{n-1}B & A_{-1}^n  \end{pmatrix}.
\end{eqnarray*}
Now, using Lemma~\ref{lemma}, we obtain Condition 2.
\end{proof}
The necessary conditions of exact null controllability characterize in fact the complete stabilizability property. 

\begin{theorem}{}\label{th:com_tsab}
  System~(\ref{eq:1}) is  completely stabilizable by a feedback law of the form
\begin{equation}\label{fbl}
 u(t)=F_{-1}\dot z(t-1) + Fz_t(\cdot),
\end{equation}
where 
$$
Fz_t(\cdot)= \int^0_{-1}\left [F_2(\theta)\dot z(t+\theta) + F_3(\theta) z(t+\theta)\right ]{\mathrm d}\theta
$$
if and only if 
\begin{enumerate}
\item $\mathrm{rank}\begin{pmatrix}
                              \Delta_\A(\lambda) & B
                             \end{pmatrix} = n $ for all $\lambda \in \Cset$,
 \item $\mathrm{rank}\begin{pmatrix}
                              \mu I - A_{-1} & B
                             \end{pmatrix} = n $ for all $\mu \in \Cset$, $\mu \ne 0$.
\end{enumerate}
\end{theorem}
\begin{proof}{}
We give a short and direct proof of the ne\-ces\-sity even if it may be obtained from Corollary 5.1.3 of \cite{Salamon_Pitman_1984}.

If Condition 1 is not verified, then there is an eigen\-value $\lambda_0$ of the operator $\A$ which cannot be modified by the control operator $\B$. 
This implies the lack of complete stabilizability. 

If Condition 2 is not verified, then there is a non\-zero eigenvalue $\mu_0$ of the matrix $A_{-1}$ which cannot be modified. Then the spectral set
$$
\{\ln |\mu_0|+ {\rm i}(\arg \mu_0 +2\pi
k)+{O}(1/k), \ 
k \in \mathbb{Z}\}\subset \sigma({\mathcal A}),
$$
 which belongs to a vertical strip, cannot also be modified. This means that complete stabilizability is not possible.

Let us show now that the two conditions are sufficient for complete stabilizability   feedback laws of the form (\ref{fbl}).

Suppose that Condition 2 is satisfied. Let us fix an arbitrary $\omega>0$. As all the non-zero poles of the matrix $A_{-1}$ are controllable 
by Condition 2,  then a matrix $F_{-1}$ can be found such that the spectrum $\sigma(A_{-1}+BF_{-1})$ verifies
$$
\forall \mu \in \sigma(A_{-1}+BF_{-1}),\ \mu \ne 0, \quad \ln|\mu| < -\omega.
$$
Consider now the neutral type system
\begin{equation}\label{sysFo}
 \dot z(t) =(A_{-1}+BF_{-1})\dot z(t-1) +Lz_t +Bu.
\end{equation}
Let $\A_1$ be the generator of system (\ref{sysFo}). From the structure of the spectrum of neutral type systems like
(\ref{eq:1}), we have only a finite number of eigenvalues $\lambda \in \sigma(\A_1)$ such that
$\real \lambda \ge -\omega$.
Now, using Condition 1 of the theorem,  a feedback  $u(t)= F_1z_t(\cdot)$, where 
\begin{eqnarray*}
 F_1z_t(\cdot)= \int^0_{-1}\left [F_2(\theta)\dot z(t+\theta) + F_3(\theta) z(t+\theta)\right ]{\mathrm d}\theta,
\end{eqnarray*}
can be found (see for example \cite{Pandolfi_1976,Pritchard_Salamon_1987,Rabah_Sklyar_Rezounenko_2008,Rabah_Sklyar_Barkhayev_2012}) 
such that all the eigenvalues $\lambda$ of the system
$$
\dot z(t) =(A_{-1}+BF_{-1})\dot z(t-1) +(L+BF_1)z_t 
$$
verify $\real \lambda < -\omega$. If we denote by $\F$ the global feedback 
$$
u(t)= F_{-1} \dot z(t-1) + F_1z_t(\cdot)),
$$
then we obtain
$$
\left \Vert \e^{(\A+\B\F)t} \right \Vert \le M \e^{-\omega t}, \quad M \ge 1.
$$
Since $\omega$ has been arbitrarily taken, this means that the system is completely stabilizable by a
feedback of  the form (\ref{fbl}).
\end{proof}
A similar result was obtained for modal controllability (assignment of characteristic quasi-polynomial) for neut\-ral sys\-tem 
with multiple discrete delays in \cite{Metelskiy_Khartovskii_2016}.

In view of Corollary~\ref{cor:1}, Theorem~\ref{th:main} and  Theorem~\ref{th:com_tsab}, one can formulate 
the following natural conjecture.

\noindent \textbf{Conjecture.} {\em System~(\ref{eq:1}) is exactly null controllable if the following two  conditions hold
\begin{enumerate}
\item $\mathrm{rank}\begin{pmatrix}
                              \Delta_\A(\lambda) & B
                             \end{pmatrix} = n $ for all $\lambda \in \Cset$,
 \item $\mathrm{rank}\begin{pmatrix}
                              \mu I - A_{-1} & B
                             \end{pmatrix} = n $ for all $\mu \in \Cset$, $\mu \ne 0$.
\end{enumerate}
This means that exact null controllability is equivalent to complete stabilizability for neutral type systems.}

It is well known that the Conjecture is verified for some class of neutral type systems with discrete del\-ays \cite{Metelskiy_Minyuk_2006,Khartovskii_Pavlovskaya_2013}
and in the case of retarded systems \cite{Olbrot_Pandolfi_1988}.
It seems to us, that one can use the conditions of  complete stabilizability to show the result of the conjecture. But at this moment, 
we have not a satisfactory formal proof. 
\section{Final exact observability}\label{observ}
The dual notion of exact null controllability in Hilbert space is the notion of final continuous 
observability. Sometimes the term continuous is replaced (by analogy) by the term  exact. In \cite{Rabah_Sklyar_2016},
the duality between exact controllability and exact observability was analyzed. In the 
present section we give the result for null exact controllability and the corresponding notion of observability.

 We consider the finite dimensional observation
\begin{equation}\label{output1}
 y(t)= \C x(t), 
\end{equation}
where $\C$ is a linear operator and $y(t) \in \Rset^p$ is a finite dimensional output. There are several
ways to design the output operator $\C$ \cite{Salamon_Springer_1983,Salamon_Pitman_1984,Metelskiy_Minyuk_2006}.
One of our goals in this paper is to investigate how to design
a minimal output operator like
\begin{equation}\label{output2}
 \C x(t)= Cz(t) \qquad \mbox{or} \qquad  \C x(t)= Cz(t-1),
\end{equation}
where $C$ is a $p\times n$ matrix. More general outputs, for example with several and/or distributed delays are not
considered here. We want to use some results on exact controllability in order to analyze, by duality,
the exact observability property in the  infinite dimensional setting like, for example, in
\cite{Tucsnak_Weiss_2009}.

The operator $\C$ defined in (\ref{output2}) is linear but not boun\-ded in $M_2$.
However, in both cases it is admissible in the following sense:
$$
\int_0^T\left \Vert\C \e^{\A t}x_0\right \Vert_{\Rset^p}^2\dd t \le \kappa^2 \|x_0\|_{M_2}^2, \qquad \forall
x_0 \in D({\mathcal A}),
$$    
because $\e^{\A t}x_0 \in D({\mathcal A}), \ t \ge 0$ (see for example \cite{Pazy_1983}).
\begin{definition}{}\label{def1}
Let  $\K$ be the output operator
$$
\K: M_2 \longrightarrow L_2(0,T;\Rset^p), \quad x_0 \longmapsto \K x_0=\C  \e^{\A t}x_0.
$$
System  (\ref{eq:1})  is said to be  exactly finally observable or contin\-uously finally observable \cite{Salamon_Pitman_1984}  if
\begin{eqnarray}\label{obsexact}
\Vert \K x_0 \Vert_{L_2}^2=
\int_0^T\left\Vert \C  \e^{\A t}x_0 \right \Vert_{\Rset^p}^2\dd t
\ge
 \gamma^2 \left\Vert  \e^{\A T}x_0\right\Vert_{M_2}^2, \hskip -5pt
\end{eqnarray}
for some constant $\gamma > 0$, and for all $x_0 \in D(\A)$. We say that the system is exactly (or continuously) observable if in (\ref{def1}) in  the 
second term of the inequality $\e^{\A T}x_0$ is replaced by $x_0$.
\end{definition}
Exact observability means that one continuously deter\-minates the initial state $z_0(\cdot)$ from the observation on $[0,T]$, final exact observability
that we can continuously determinate the final state $z_T(\cdot)$.
\medskip

The exact (final) observability depends essentially on the topology of the space.
We can expect that, the given neutral type system is not exactly observable
if we consider $x_0 \in D(\A)$,  with the norm of the graph and no longer in the topology of $M_2$.
In fact,  we obtain the final observability in the initial norm but we need some 
delay in the observation in the general case.

In order to use the duality between observability and controllability, we need the expression of the 
adjoint operator $\K^*$ in the duality  with respect to the pivot space $M_2$ in the embedding
$$
X_1 \subset X=M_2 \subset X_{-1},
$$
where $X_1= D(\A)$ with the graph norm noted $\|x\|_1$ and
$X_{-1}$ the completion of the space $M_2$ with respect to the resolvent norm
$\|x\|_{-1}=\left \|(\lambda I- \A)^{-1}x\right\|_{M_2}$.
The duality relation is
$$
\left \langle \K x_0, u(\cdot)\right \rangle_{L_2(0,T;\Rset^p)}
= \left \langle  x_0,\K^* u(\cdot)\right \rangle_{X_{1},X_{-1}^\dd},
$$
where $X_{-1}^\dd$ is constructed as $X_{-1}$ with $\A^*$ instead of $\A$
(see  \cite{Tucsnak_Weiss_2009} for example): $X_{-1}^\dd$ is the completion of the space  $M_2$ with the resolvent norm corresponding to 
the operator ${\mathcal A}^*$.

Exact null controllability is dual  with exact final observability  
in the corresponding spaces and with the corresponding topologies.
It is expected that the operator $\K^*$ corresponds to a control operator for some adjoint system. 
However, the situation is not so simple, as it was pointed out  in the paper \cite{Rabah_Sklyar_2016}, from which we
take our main considerations on duality.
 \begin{proposition}{}{\em \cite{Rabah_Sklyar_Rezounenko_2008,Rabah_Sklyar_2016}} \label{prop1}
The adjoint operator  ${\mathcal A}^{*}$ is
given by
$$
{\mathcal A}^{*}
\begin{pmatrix}
w\cr \psi (\cdot)
\end{pmatrix}
= \begin{pmatrix}
( A_2^{*}(0)w + \psi (0)
\cr -\frac{\dd [ \psi(\theta)+A_2^{*}(\theta)w ]}
{\dd \theta}
 +A_3^{*}(\theta)w
\end{pmatrix},
$$
with the domain $D({\A}^{*})$ consisting of  $(w , \psi(\cdot)) \in M_2$ such that:
$$
  \left\{ 
\begin{array}{l}
\psi(\theta)+A_2^{*}(\theta)w \in H^1,\cr
A_{-1}^{*} \left ( A_2^{*}(0)w + \psi(0)\right)
 =\psi(-1)+A_2^{*}(-1)w.
\end{array}
\right.
$$
Let $x$ be a solution of  the abstract equation
\begin{equation}\label{eq:2.0*}
 \dot x=\A^*x, \quad x(t)=
\begin{pmatrix}
w(t)
\cr \psi_t(\theta)
\end{pmatrix}.
\end{equation}
Then the function $w(t)$ is the solution of the neutral type equation\\[2ex]
$
\dot w(t+1) = A_{-1}^*\dot w(t) + 
$
 \begin{eqnarray}\label{eq:adj0}
  \int_{-1}^0 \left [A_2^*(\tau)\dot w(t+1+\tau) +
 A_3^*(\tau) w(t+1+\tau)\right ]\dd \tau. 
\end{eqnarray}
\end{proposition}

\medskip

This means that the form of the adjoint system is not a simple transposition of the initial one (\ref{eq:1}). 
Let us now specify  the relation between the solutions of the neutral type equation
(\ref{eq:adj0}) related to the adjoint system~(\ref{eq:2.0*}) and the transposed neutral type equation\\[2ex]
$ \dot z(t) = A_{-1}^*\dot z(t-1) +$
\begin{eqnarray}\label{transposed}
  \int_{-1}^0 \left [A_2^*(\tau)\dot z(t+\tau),
 A_3^*(\tau) z(t+\tau)\right ]\dd \tau , 
\end{eqnarray}
with initial $z_0(\theta)$. Let $\A^\dag$ be the infinitesimal generator of the semigroup 
corresponding to equation (\ref{transposed}).

Let us put
$$
\begin{pmatrix}
w(t) \cr \psi_t(\theta)
\end{pmatrix} =
 \e^{{\A^*} t}\xi_0=
\e^{{\A^*} t} \begin{pmatrix}
w(0) \cr
\psi_0(\theta)
\end{pmatrix},
$$
and the conditions of 
$$
\begin{pmatrix}
v(t) \cr z_t(\theta)
\end{pmatrix}
 =
\begin{pmatrix}
w(t+1)-A_{-1}^*w(t) \cr w(t+1+\theta)
\end{pmatrix}=
\e^{{\A^\dag} t}
\begin{pmatrix}
v(0) \cr z_0(\theta)\end{pmatrix}
,
$$
where $z_0(\theta)= w(\theta+1)$ and $v(0)=z_0(0)-A_{-1}z_0(-1)$. We can 
give the explicit relation between the initial conditions $ \xi_0$ and
$x_0$:
$$
 \xi_0 = \begin{pmatrix}
w(0) \cr \psi_0(\theta)
\end{pmatrix}, \qquad x_0= \begin{pmatrix}
v(0) \cr z_0(\theta)
\end{pmatrix}.
$$
The formal  relation between these vectors is
$$
 \xi_0 = \begin{pmatrix}
w(0) \cr \psi_0(\theta)
\end{pmatrix}= \Phi x_0= \Phi
\begin{pmatrix}
w(1)-A_{-1}w(0) \cr w(\theta+1)
\end{pmatrix},
$$
and we have the following result.
\begin{theorem}{\cite{Rabah_Sklyar_2016}}\label{initial}
The operator $\Phi$ representing the relation between initial conditions $x_0$ and $\xi_0$
corresponding to  neutral type systems (\ref{eq:2.0*}) -- (\ref{eq:adj0}) and (\ref{transposed}) is linear bounded
and bounded invertible from $X_1^\mathrm{d}$ to $M_2$, where $X_1^\mathrm{d}$ is $D(\A^*)$ 
with the graph norm.
\end{theorem}

Let us now consider the reachability  operator of the transposed controlled system:
$$
 \dot x (t)= \A^\dag x(t)+ \C^\dag u(t),
$$
where $ \C^\dag = 
\begin{pmatrix}
C^*\cr 0
\end{pmatrix}$. This operator is given by
$$
\R_T^\dag u(\cdot) = \int_0^T \e^{\A^\dag(T-\tau)} \C^\dag u(\tau)\dd \tau.      
$$
The operator $\K$ may be written using $\R_T^\dag$ and the semi\-group $\e^{\A^\dag}$ of  system (\ref{transposed})
 as follows (see \cite{Rabah_Sklyar_2016}):
\begin{eqnarray}\label{Kdual}
\K x_0= \left \{
\begin{array}{lcl}
\R_T^{\dag *}\Phi x_0 & \mbox{if}  & \C 
x(t)= Cz(t-1),\\[1ex]
\R_T^{\dag *}\e^{\A^{\dag *}}\Phi x_0 & \mbox{if}  &\C x(t)= Cz(t).
\end{array}
\right.
\end{eqnarray}
We can now formulate the main result of this section.
\begin{theorem}{}\label{th:fobs}
System (\ref{eq:1}) with the output $y=Cz(t-1)$ is exactly (continuously)  finally observable if and only if
system (\ref{transposed}) is exactly null controllable. A necessary condition of exact final observability is given 
by two conditions:
\begin{enumerate}
 \item $\Ker \begin{pmatrix}
              \Delta_\A(\lambda) \cr C
             \end{pmatrix}= \{0\}$ for all $\lambda \in \Cset$,
\item $\Ker \begin{pmatrix}
              \lambda I -A_{-1} \cr C
             \end{pmatrix}= \{0\}$ for all $\lambda \in \Cset$, $\lambda \ne 0$.
\end{enumerate}
\end{theorem}
\begin{proof}{}
 According to  relation (\ref{Kdual}) we have
$$
\Vert \K x_0 \Vert_{L_2} = \left (\int_0^T\left \Vert C^*\e^{\A^{\dag *}(T-\tau)} \Phi x_0  \right \Vert^2 \dd \tau \right)^{\frac{1}{2}}.
$$
As  system (\ref{transposed}) is exactly null controllable, we obtain
$$
\Vert \K x_0 \Vert_{L_2} \ge \delta \left \Vert \e^{\A^{\dag *}T} \Phi x_0 \right \Vert,
$$
for all  $x_0 \in D(\A)$. It is easy to see from \cite{Rabah_Sklyar_2016} that
$$
\e^{\A^{\dag *}T} \Phi x_0= \Phi  \e^{\A^{\dag *}(T-\tau)} x_0 = \e^{\A T}x_0.
$$
This gives
$$
\Vert \K x_0 \Vert_{L_2} \ge  \delta \Vert \e^{\A T} x_0 \Vert,
$$
which means that exact final observability holds.
\end{proof}
For the case of the output $y=Cz(t)$ we cannot say anything if $\det (\A_{-1}) = 0$. If 
$\A_{-1}$ is not singular, then $\e^{\A t}$ is a group and exact final observability coincides with
exact observability  \cite{Rabah_Sklyar_2016}.
\section{Examples}
To illustrate our results and hypothesis we give here 3 examples. The first one shows that for continuous obser\-vability a delay in the output is needed
if the semigroup is not a group. The second one is taken from \cite{Metelskiy_Minyuk_2006} and it is shown that in fact we have exact controllability (not only exact 
null controllability).
The last example illustrates our Conjecture on equivalence between exact controllability and complete stabilizability.

All   examples are given in the form of a system with one discrete delay:
\begin{equation}\label{examp}
\dot z (t)= A_{-1} z(t-1)+ A_0 z(t) + A_1 z(t-1) + Bu(t).
\end{equation}
\subsection*{Example 3.} System (\ref{examp}) with
$$
A_0= \begin{pmatrix}0 & 1 \cr 0&0\end{pmatrix},\  A_1= 0, \ A_{-1}=\begin{pmatrix}
                                                     0 & 1 \cr 0&0
\end{pmatrix}, \ B=\begin{pmatrix}0 \cr 1\end{pmatrix}.
$$
It is easy to see that, for all $ \lambda \in \Cset$,
$$
  \mathrm{rank} \begin{pmatrix}\Delta_{\A}(\lambda) & B\end{pmatrix}=
 \begin{pmatrix}\lambda & -\lambda \e^{-\lambda }- 1 & 0 \cr
0&\lambda & 1 \end{pmatrix} = 2.
$$
Moreover,  for all $\lambda \in \Cset$, 
$\mathrm{rank} \begin{pmatrix}\lambda I - A_{-1} & B \end{pmatrix}=n$, then the  system is exactly controllable (not only to zero).
The transposed system
$$
\left \{
\begin{array}{rcl}
 \dot z_1(t) &= &0\cr
\dot z_2(t)& =& \dot z_1(t-1) + z_1(t)
\end{array}\right.
$$
is continuously observable with the output $y=z_2(t-1)$ but not with $y(t)=z_2(t)$.

\subsection*{Example 4.} 
The following  system was given for exact null controllability and continuous final 
observability in \cite{Metelskiy_Minyuk_2006}:
$$
A_0=0, \ A_1= \begin{pmatrix}0 & 1 \cr 0&0\end{pmatrix}, \  A_{-1}=
\begin{pmatrix}0 & -1 \cr 0&1\end{pmatrix}.
$$
In fact, for this system the initial condition is exactly observable by the output 
$$
y=Cz(t-1), \quad C=\begin{pmatrix}1 & 0\end{pmatrix},
$$ 
and the transposed system is exactly controllable because, for all $\lambda \in \Cset$,
$$
 \mathrm{rank} \begin{pmatrix}\lambda I - A_{-1}^* & C^*\end{pmatrix} = 
\mathrm{rank}\begin{pmatrix}\lambda & 0 & 1 \cr 1  & \lambda -1 & 0\end{pmatrix} = 2. 
$$
However, the initial system  is not exactly observable by the output  $y= Cz(t)= z_1(t)$, because the initial function 
$z_0(\theta), \theta \in [0,1[$ cannot be determined.
\subsection*{Example 5.} System (\ref{examp}) with
$$
A_0= \begin{pmatrix}0 & 0 \cr 1&0
      \end{pmatrix}
, \ A_1= 0, \  A_{-1}=\begin{pmatrix}1 & 0 \cr 0&0
                            \end{pmatrix},
\ B=\begin{pmatrix}1 \cr 0
         \end{pmatrix}.
$$
We have for all $\lambda \in \Cset$,
$$
 \mathrm{rank} \begin{pmatrix}
\Delta_{\A}(\lambda) & B \end{pmatrix}
=
\begin{pmatrix}
\lambda  - \lambda \e^{-\lambda } & 0& 1 \cr
-1 &\lambda & 0
 \end{pmatrix} = 2,
$$
and for all complex  $\lambda \ne 0$,
$$
 \mathrm{rank} \begin{pmatrix}\lambda I - A_{-1} & B
 \end{pmatrix}
 = 
\begin{pmatrix}\lambda -1 & 0 & 1\cr
0 &\lambda & 0
\end{pmatrix}= 2. 
$$
The system is exactly null controllable by Lemma~\ref{lemma} and result in \cite{Metelskiy_Minyuk_2006}. It is completely stabilizable by Theorem~\ref{fbl}. 
Consider now the transposed system 
$$
\left \{
\begin{array}{rcl}
 \dot z_1(t) &= &z_2(t),\cr
\dot z_2(t)& =& \dot z_2(t-1).
\end{array}\right.
$$
This system is continuously finally observable by the feedback $y=z_1(t-1)$ by Theorem~\ref{th:fobs}.

\section{Conclusion} We gave some relations  between exact null controllability and complete stabilizability
of abstract systems in Hilbert spaces.  A characterization of complete stabilizability  has been given for a large class of
linear neutral type systems. Necessary conditions of exact null controllability are given, which conjectured to be also 
sufficient for neutral type systems, even if they are not in the general case. 
This also enables    the final continuous observability of such systems to be characterized.
The following step is to prove the conjecture and to extend such results to the problem of detectability, which is dual with stabilizability. 
 \addtolength{\textheight}{-5ex}   

\begin{center}
 \textbf{Acknowledgment}
\end{center}

An earlier version of this work was presented  at the International Conference
 ``System Analysis: Mod\-elling an Con\-trol'', 
organized in the memory of Ac\-ademician Arkady Kryazhimskiy at the Krasovskii Ins\-titute of Math\-ematics, 
Ural Branch of RAS, Ekaterinburg, Russia. The first author thanks Professor Vyacheslav Maksimov 
and all the Organizing Com\-mittee of the Conference.

\begin{thebibliography}{10}

\bibitem{Curtain_L_T_Z_1997}
Ruth~F. {Curtain}, Hartmut {Logemann}, Stuart {Townley}, and Hans {Zwart}.
\newblock {Well-posedness, stabilizability, and admissibility for
  Pritchard-Salamon systems.}
\newblock {\em {J. Math. Syst. Estim. Control}}, 4(4):493--496, 1997.

\bibitem{Curtain_Zwart_1995}
Ruth~F. Curtain and Hans Zwart.
\newblock {\em An introduction to infinite-dimensional linear systems theory},
  volume~21 of {\em Texts in Applied Mathematics}.
\newblock Springer-Verlag, New York, 1995.

\bibitem{Douglas_1966}
R.~G. Douglas.
\newblock On majorization, factorization, and range inclusion of operators on
  {H}ilbert space.
\newblock {\em Proc. Amer. Math. Soc.}, 17:413--415, 1966.

\bibitem{Dusser_Rabah_2001}
Xavier Dusser and Rabah Rabah.
\newblock On exponential stabilizability of linear neutral systems.
\newblock {\em Math. Probl. Eng.}, 7(1):67--86, 2001.

\bibitem{Guo_Z_H_2003}
Faming {Guo}, Qiong {Zhang}, and Falun {Huang}.
\newblock {Well-posedness and admissible stabilizability for Pritchard-Salamon
  systems.}
\newblock {\em {Appl. Math. Lett.}}, 16(1):65--70, 2003.

\bibitem{Hale_Verduyn_1993}
Jack~K. Hale and Sjoerd~M. {Verduyn Lunel}.
\newblock {\em Introduction to functional-differential equations}, volume~99 of
  {\em Applied Mathematical Sciences}.
\newblock Springer-Verlag, New York, 1993.

\bibitem{Hale_Verduyn_2002}
Jack~K. Hale and Sjoerd~M. {Verduyn Lunel}.
\newblock Strong stabilization of neutral functional differential equations.
\newblock {\em IMA J. Math. Control Inform.}, 19(1-2):5--23, 2002.
 

\bibitem{Ito_Tarn_1985}
Kazufumi Ito and T.~J. Tarn.
\newblock A linear quadratic optimal control for neutral systems.
\newblock {\em Nonlinear Anal.}, 9(7):699--727, 1985.

\bibitem{Khartovskii_Pavlovskaya_2013}
V.~E. Khartovski{\u\i} and A.~T. Pavlovskaya.
\newblock Complete controllability and controllability for linear autonomous
  systems of neutral type.
\newblock {\em Automation and Remote Control}, 74(5):769--784, 2013.

\bibitem{Kuperman_Repin_1971}
L.~M. Kuperman and Ju.~M. Repin.
\newblock On the question of controllability in infinite-dimensional spaces.
\newblock {\em Dokl. Akad. Nauk SSSR}, 200:767--769, 1971. English translation in 
{\em Soviet Mathematics--Doklady} {\bf 12}(5):~1469--1472.

\bibitem{Louis_Wexler_1983}
J.-C. Louis and D.~Wexler.
\newblock On exact controllability in {H}ilbert spaces.
\newblock {\em J. Differential Equations}, 49(2):258--269, 1983.

\bibitem{Megan_1975} M. Megan. 
\newblock On the stabilizability and controllability of linear dissipative systems in Hilbert space, 
{\em Seminarul de Ecuatii Functional, Universitatea din Timisoara}, {\bf 32}(2016):~1--15. 

\bibitem{Metelskiy_Khartovskii_2016}
A.~V. Metel'skii and V.~E. Khartovskii.
\newblock Criteria for modal controllability of linear systems of neutral type.
\newblock {\em Differ. Equ.}, 52(11):1453--1468, 2016.
\newblock Translation of Differ. Uravn. {{\bf{5}}2} (2016), no. 11, 1506--1521.

\bibitem{Metelskiy_Minyuk_2006}
A.~V. Metel{\cprime}ski{\u\i} and S.~A. Minyuk.
\newblock Criteria for the constructive identifiability and complete
  controllability of linear time-independent systems of neutral type.
\newblock {\em Izv. Ross. Akad. Nauk Teor. Sist. Upr.}, (5):15--23, 2006. (Russian). English translation in 
{\em Journal of Computer and System Sciences International} {\bf 45}(5):~690–-698.

\bibitem{Michiels_Niculescu_2007}
Wim Michiels and Silviu-Iulian Niculescu.
\newblock {\em Stability and stabilization of time-delay systems. An
  eigenvalue-based approach}, volume~12 of {\em Advances in Design and
  Control}.
\newblock Society for Industrial and Applied Mathematics (SIAM), Philadelphia,
  PA, 2007.

\bibitem{Oconnor_Tarn_1983}
Daniel~A. O'Connor and T.~J. Tarn.
\newblock On stabilization by state feedback for neutral differential
  equations.
\newblock {\em IEEE Trans. Automat. Control}, 28(5):615--618, 1983.

\bibitem{Pandolfi_1976}
L.~Pandolfi.
\newblock Stabilization of neutral functional differential equations.
\newblock {\em J. Optimization Theory Appl.}, 20(2):191--204, 1976.

\bibitem{Olbrot_Pandolfi_1988}
L.~Pandolfi and A.~W. Olbrot.
 \newblock 
Null controllability of a class of functional-differential systems,
{\em International Journal of Control} {\bf 47}(1):~193--208.

\bibitem{Pazy_1983}
A.~Pazy.
\newblock {\em Semigroups of linear operators and applications to partial
  differential equations}, volume~44 of {\em Applied Mathematical Sciences}.
\newblock Springer-Verlag, New York, 1983.

\bibitem{Pritchard_Salamon_1987}
A.~J. Pritchard and D.~Salamon.
\newblock The linear quadratic control problem for infinite dimensional systems
  with unbounded input and output operators.
\newblock {\em SIAM J. Control and Optimization}, 25(1):121--144, 1987.

\bibitem{Rabah_Sklyar_2016}
R.~Rabah and G.~M. Sklyar.
\newblock Exact observability and controllability for linear neutral type
  systems.
\newblock {\em Systems Control Lett.}, 89:8--15, 2016.

\bibitem{Rabah_Sklyar_Barkhayev_2016}
R.~Rabah, G.~M. Sklyar, and P.~Yu. Barkhayev.
\newblock On exact controllability of neutral time-delay systems.
\newblock {\em Ukrainian Math. Journal}, 68(6):800--815, 2016.

\bibitem{Rabah_Sklyar_Rezounenko_2005}
R.~Rabah, G.~M. Sklyar, and A.~V. Rezounenko.
\newblock Stability analysis of neutral type systems in {H}ilbert space.
\newblock {\em J. Differential Equations}, 214(2):391--428, 2005.

\bibitem{Rabah_Sklyar_Rezounenko_2008}
R.~Rabah, G.~M. Sklyar, and A.~V. Rezounenko.
\newblock On strong regular stabilizability for linear neutral type systems.
\newblock {\em J. Differential Equations}, 245(3):569--593, 2008.

\bibitem{Rabah_Karrakchou_1997}
{R}abah Rabah and {J}amila Karrakchou.
\newblock {On Exact Controllability and Complete Stabilizability for Linear
  Systems in Hilbert Spaces}.
\newblock {\em {Applied Mathematics Letters}}, 10(1):pp. 35--40, 1997.

\bibitem{Rabah_Sklyar_2007}
{R}abah Rabah and {G}rigory~M. Sklyar.
\newblock The analysis of exact controllability of neutral-type systems by the
  moment problem approach.
\newblock {\em SIAM J. Control Optim.}, 46(6):2148--2181, 2007.

\bibitem{Rabah_Sklyar_Barkhayev_2012}
{R}abah Rabah, {G}rigory~{M.} Sklyar, and {P}avel~{Yu.} Barkhayev.
\newblock Stability and stabilizability of mixed retarded-neutral type systems.
\newblock {\em ESAIM Control Optim. Calc. Var.}, 18(3):656--692, 2012.

\bibitem{Rhandi_2002}
A. Rhandi.
\newblock Spectral theory for positive semigroups and applications, Vol.~2 of
  {\em Quaderni di Matematica}, ESE - Salento University Publishing, Lecce.
\bibitem{Richard_2003}
Jean-Pierre Richard.
\newblock Time-delay systems: an overview of some recent advances and open
  problems.
\newblock {\em Automatica J. IFAC}, 39(10):1667--1694, 2003.

\bibitem{Salamon_Springer_1983}
D.~Salamon.
\newblock Neutral functional differential equations and semigroups of
  operators.
\newblock In {\em Control theory for distributed parameter systems and
  applications ({V}orau, 1982)}, volume~54 of {\em Lecture Notes in Control and
  Inform. Sci.}, pages 188--207. Springer, Berlin, 1983.

\bibitem{Salamon_1987}
D.~Salamon.
\newblock Infinite dimensional linear systems with unbounded control and
  observation: a functional analytical approach.
\newblock {\em Trans. American Math. Soc.}, 300(2):383--431, 1987.

\bibitem{Salamon_Pitman_1984}
Dietmar Salamon.
\newblock {\em Control and observation of neutral systems}, volume~91 of {\em
  Research Notes in Mathematics}.
\newblock Pitman (Advanced Publishing Program), Boston, MA, 1984.

\bibitem{Sklyar_Szkiebiel_2013}
Grigory~M. Sklyar and Grzegorz Szkibiel.
\newblock Controlling a non-homogeneous {T}imoshenko beam with the aid of the
  torque.
\newblock {\em Int. J. Appl. Math. Comput. Sci.}, 23(3):587--598, 2013.

\bibitem{Tucsnak_Weiss_2009}
Marius Tucsnak and George Weiss.
\newblock {\em {Observation and control for operator semigroups.}}
\newblock {Birkh\"auser Advanced Texts. Basler Lehrb\"ucher. Basel:
  Birkh\"auser. xi, 483~p.}, 2009.

\bibitem{Zabczyk_1976}
J.~Zabczyk.
\newblock Remarks on the algebraic {R}iccati equation in {H}ilbert space.
\newblock {\em Appl. Math. Optim.}, 2(3):251--258, 1975/76.

\bibitem{Zabczyk_1992}
Jerzy Zabczyk.
\newblock {\em Mathematical control theory: an introduction}.
\newblock Systems \& Control: Foundations \& Applications. Birkh\"auser Boston,
  Inc., Boston, MA, 1992.

\bibitem{Zeng_Yi_Xie_2013}
Yi~Zeng, Zuoshi Xie, and Faming Guo.
\newblock On exact controllability and complete stabilizability for linear
  systems.
\newblock {\em Appl. Math. Lett.}, 26(7):766--768, 2013.

\end{thebibliography}
\def\cprime{$'$} \def\cprime{$'$}

\end{document}